\documentclass[a4paper]{amsart}
\usepackage{amsmath,amssymb,amsxtra,amsfonts,amsthm,comment,url}
\usepackage{tikz} 
\usetikzlibrary{arrows,backgrounds,decorations,automata}

\theoremstyle{plain}
\newtheorem{theorem}{Theorem}
\newtheorem{lemma}{Lemma}
\newtheorem{corollary}{Corollary}
\newtheorem{proposition}{Proposition}

\theoremstyle{definition}
\newtheorem{remark}{Remark}

\newtheorem{example}{Example}

\newcommand{\B}{\mathbb}

\newcounter{nootje}
\newcommand\noot[1]
 {\marginpar{\tiny\begin{minipage}{30mm}\begin{flushleft}\thenootje : #1\end{flushleft}\end{minipage}}\addtocounter{nootje}{1}}
\makeatletter
  \def\vhrulefill#1{\leavevmode\leaders\hrule\@height#1\hfill \kern\z@}
\makeatother

\begin{document}

\title{Transcendence tests for Mahler functions}
\author{Jason P. Bell}
\address{Department of Pure Mathematics, University of Waterloo, Waterloo, Canada}
\email{jpbell@uwaterloo.ca}

\author{Michael Coons}
\address{School of Math.~and Phys.~Sciences\\
University of Newcastle\\
Callag\-han\\
Australia}
\email{Michael.Coons@newcastle.edu.au}

\thanks{The research of J.~P.~Bell was supported by NSERC grant 31-611456 and the research of M.~Coons was supported by ARC grant DE140100223.}

\date{\today}

\keywords{Transcendence, Mahler functions, radial asymptotics}
\subjclass[2010]{Primary 11J91; Secondary 39A06, 30B30}%

\begin{abstract} We give two tests for transcendence of Mahler functions. For our first, we introduce the notion of the eigenvalue $\lambda_F$ of a Mahler function $F(z)$, and develop a quick test for the transcendence of $F(z)$ over $\mathbb{C}(z)$, which is determined by the value of the eigenvalue $\lambda_F$. While our first test is quick and applicable for a large class of functions, our second test, while a bit slower than our first, is universal; it depends on the rank of a certain Hankel matrix determined by the initial coefficients of $F(z)$. We note that these are the first transcendence tests for Mahler functions of arbitrary degree. Several examples and applications are given.
\end{abstract}

\maketitle

\section{Introduction}

A function $F(z)\in\B{C}[z]$ is called a {\em $k$-Mahler function} provided there are integers $k\geqslant 2$ and $d\geqslant 1$ such that \begin{equation}\label{MFE}a_0(z)F(z)+a_1(z)F(z^k)+\cdots+a_d(z)F(z^{k^d})=0,\end{equation} for some polynomials $a_0(z),\ldots,a_d(z)\in\B{C}[z]$. The minimal such positive integer $d$ is called the {\em degree} of the $k$-Mahler function $F(z)$.

These functions were introduced by Mahler \cite{M1929, M1930a, M1930b}, who was interested in the transcendence of several specific Mahler functions and their special values. The classical example is the Fredholm number $\sum_{n\geqslant 0} 2^{-2^n},$ the transcendence of which follows from the transcendence of the associated function $A(z)=\sum_{n\geqslant 0}z^{2^n}$ that satisfies the Mahler functional equation $$zA(z)-(z+1)A(z^2)+A(z^4)=0.$$ The function $A(z)$ is the canonical example of a function that has the unit circle as a natural boundary, and so it is transcendental; indeed, this example is included in Whittaker and Watson's classic text, {\em A Course of Modern Analysis} \cite[Section~5$\cdot$501]{WW1902}, which first appeared well-before the work of Mahler.

In a step towards classifying the diffeo-algebraic nature of Mahler functions, B\'ezivin \cite{B1994} showed that a Mahler function, which satisfies a homogeneous linear differential equation, is necessarily rational. Since algebraic functions satisfy such differential equations, his result shows that an irrational Mahler function cannot be algebraic. Thus in order to determine the transcendence of a Mahler function, one need only show irrationality. 

Until now, only ad hoc methods have been used to prove the transcendence of various Mahler functions. In this short paper, we rectify this situation by providing two tests for transcendence of Mahler functions. Our first test is formulated by associating a characteristic polynomial to a Mahler function. If this characteristic polynomial has distinct nonzero roots (in $\B{C}$), then there is a root $\lambda_F$ that is naturally associated to $F(z)$, which we call the {\em eigenvalue of $F(z)$}. Our first test is stated in Figure \ref{Fig}.

\begin{figure}[h]
\noindent\vhrulefill{1pt}

\noindent\textbf{Eigenvalue test for transcendence of Mahler functions.}

\vspace{-.2cm}
\noindent\vhrulefill{1pt}
\begin{flushleft}

\noindent Let $k\geqslant 2$ and $d\geqslant 1$ be  integers and $F(z)$ be a $k$-Mahler function converging inside the unit disc satisfying $$a_0(z)F(z)+a_1(z)F(z^k)+\cdots+a_d(z)F(z^{k^d})=0,$$ for polynomials $a_0(z),\ldots,a_d(z)\in\B{C}[z]$. Set $a_i:=a_i(1)$ and form the polynomial $$p_F(\lambda):=a_0\lambda^d+a_1\lambda^{d-1}+\cdots+a_{d-1}\lambda+a_d.$$

\vspace{.2cm}
\noindent If $a_0a_d\neq 0$ and $p_F(\lambda)$ has distinct roots, then the function $F(z)$ is transcendental over $\B{C}(z)$ provided
\vspace{.2cm}
\begin{itemize}
\item $p(k^n)\neq 0$ for all $n\in\B{Z}$ or
\item the eigenvalue $\lambda_F\neq k^n$ for any $n\in\B{Z}$.
\end{itemize}

\vspace{.2cm}
\noindent If $\lambda_F=k^n$ for some $n\in\B{Z}$, the test is inconclusive.
\end{flushleft}

\noindent\vhrulefill{1pt}
\vspace{-.2cm}
\caption{Eigenvalue test for transcendence of Mahler functions.}
\label{Fig}
\end{figure}

While our first test can be completed quickly, it applies only to Mahler functions satisfying certain conditions, such as analyticity in the unit disc. Our second test is unconditional; it is stated in Figure \ref{Fig2}.

\begin{figure}[h]
\noindent\vhrulefill{1pt}

\noindent\textbf{Universal test for transcendence of Mahler functions.}

\vspace{-.2cm}
\noindent\vhrulefill{1pt}
\begin{flushleft}

\noindent Let $k\geqslant 2$ and $d\geqslant 1$ be  integers and $F(z)$ be a $k$-Mahler function satisfying $$a_0(z)F(z)+a_1(z)F(z^k)+\cdots+a_d(z)F(z^{k^d})=0,$$ for polynomials $a_0(z),\ldots,a_d(z)\in\B{C}[z]$. Set $H:=\max\{\deg a_i(z): i=0,\ldots,d\}$ and $$\kappa:=\lfloor H(k-1)/(k^{d+1}-2k^d+1)\rfloor+\lfloor H/k^{d-1}(k-1)\rfloor+1.$$

\vspace{.2cm}
\begin{enumerate}
\item[Step 1.]  Compute the coefficient, $f(i)$, of $z^i$ of $F(z)$ for $$i=0,1,\ldots ,\kappa + H + \kappa (k^{d+1}-1)/(k-1).$$ 
\item[Step 2.] Form the $$(1+\kappa)\times (1+H + \kappa (k^{d+1}-1)/(k-1))$$ matrix ${\bf M}$ whose $(i,j)$-entry if $f(i+j-2)$.
\vspace{.1cm}
\item[Step 3.] Put this matrix in echelon form and verify whether it has full rank (i.e., rank
equal to $1+\kappa$).
\vspace{.1cm}
\item[Step 4.] If it does, then $F(z)$ is transcendental; otherwise it is rational.  
\end{enumerate} 
\end{flushleft}

\noindent\vhrulefill{1pt}
\vspace{-.2cm}
\caption{Universal test for transcendence of Mahler functions.}
\label{Fig2}
\end{figure}

This paper contains three further sections. In Section \ref{sec2}, we prove the validity of our eigenvalue transcendence test via a result on the radial asymptotics of Mahler functions, which itself is of independent interest. In Section \ref{sec3}, we focus on our universal transcendence test. In the last section, we list a few examples and remarks about the our two tests.

\begin{remark}
Our transcendence tests should be compared with two earlier results of Coons, \cite[Theorem 2.2]{C2012} and \cite[Theorem 3.1]{C2013}, which established transcendence tests for very specific families of Mahler functions of degree one and two.
\end{remark}

\section{Eigenvalues and transcendence of Mahler functions}\label{sec2}

Suppose that $F(z)$ is a $k$-Mahler function satisfying \eqref{MFE} that converges inside the unit disc and for $i=0,\ldots,d$ define $a_i:=a_i(1)$. Then the function $f(x):=F(z^{k^d})=F(e^{-k^dt})$, where $z=e^{-t}$ and $t=k^{-x}$,  satisfies the Poincar\'e difference equation $$\left(a_0+O(k^{-x})\right)f(x+d)+\cdots+\left(a_{d-1}+O(k^{-x})\right)f(x+1)+\left(a_d+O(k^{-x})\right)f(x)=0,$$ which has limiting characteristic polynomial \begin{equation}\label{LCP} p_F(\lambda):=a_0\lambda^d+a_1\lambda^{d-1}+\cdots+a_{d-1}\lambda+a_d.\end{equation}
If $a_0a_d\neq 0$ and $\lambda_1,\ldots,\lambda_d$ are distinct, a direct application of a theorem of Evgrafov~\cite{E1958} (see also \cite{E1953} for background) gives the existence of an eigenvalue $\lambda_F\in\{\lambda_1,\ldots,\lambda_d\}$ such that $$f(x)=\tilde{C}\lambda_F^x (1+o(1))$$ as $x\to\infty$ along $x\equiv x_0 (\bmod\ \B{Z})$ for some $\tilde{C}=\tilde{C}(x_0)\neq 0$. Note that $\tilde C(x)$ is a $1$-periodic real-analytic function in an interval $x>\sigma_0$ because of the analytic dependence of the solution of the difference equation on the initial data. 

This implies that as $t\to 0^+$ $$F(e^{-t})=\frac{\hat C(t)}{t^{\log_k \lambda_F}}(1+o(1)),$$ and further \begin{equation*} F(z)=\frac{C(z)}{(1-z)^{\log_k \lambda_F}}(1+o(1)),\end{equation*} as $z\to1^-$, where $C(z)$ is real-analytic and satisfies $C(z)=C(z^k)$ for $z\in(0,1)$. Here we have used $\log_k$ to denote the principal value of the base-$k$ logarithm.

In the notation of the previous paragraph, we have proved the following theorem.

\begin{theorem}\label{asymp} Let $F(z)$ be a $k$-Mahler function satisfying \eqref{MFE} whose characteristic polynomial $p_F(\lambda)$ has distinct roots. Then there is an eigenvalue $\lambda_F$ with $p_F(\lambda_F)=0$, such that as $z\to 1^-$ $$F(z)=\frac{C(z)}{(1-z)^{\log_k \lambda_F}} (1+o(1)),$$ where $\log_k$ denotes the principal value of the base-$k$ logarithm and $C(z)$ is a real-analytic nonzero oscillatory term, which on the interval $(0,1)$ is bounded away from $0$ and $\infty$, and satisfies $C(z)=C(z^k)$. 
\end{theorem}

\begin{remark} Questions about the asymptotic behaviour of Mahler functions are quite classical, and some special cases of Theorem \ref{asymp} are known. See, e.g., Mahler~\cite{M1940}, de Bruijn \cite{dB1948}, Dumas \cite{D1993These}, Dumas and Flajolet \cite{DF1996}, and most recently Brent, Coons, and Zudilin \cite{BCZ2015}.
\end{remark}

Our next result proves the validity of our eigenvalue transcendence test; it is a near-immediate corollary of Theorem~\ref{asymp}. 

\begin{theorem}\label{main} Let $F(z)$ be a $k$-Mahler function converging inside the unit disc. If the eigenvalue $\lambda_F$ exists and is not an integral power of $k$, then $F(z)$ is transcendental over $\B{C}(z)$.
\end{theorem}

\begin{proof} We use B\'ezivin's result (discussed in the Introduction), so that we only need prove the irrationality of $F(z)$.

To this end, suppose that $F(z)$ is a $k$-Mahler function converging inside the unit disc, and that the eigenvalue $\lambda_F$ exists and is not an integral power of $k$. Then Theorem~\ref{asymp} implies that $$0<\liminf_{z\to 1^{-}}\ (1-z)^{\beta} F(z)\leqslant \limsup_{z\to 1^{-}}\ (1-z)^{\beta} F(z)<\infty,$$ where $\beta:=\log_k\lambda_F$ is not integral. This immediately implies that $F(z)$ is irrational, as a rational function can have only integral order zeros and poles, and cannot exhibit strange oscillatory behaviour. 
\end{proof}

Theorem \ref{asymp} asserts the existence of an eigenvalue associated to $F(z)$, which we call $\lambda_F$, and Theorem \ref{main} uses the value of $\lambda_F$ to give a transcendence result. Our next result of this section gives a way to calculate $\lambda_F$, so that one may apply the above results to specific examples.

\begin{proposition}\label{prop} Let $F(z)$ be a $k$-Mahler function, converging inside the unit disc, for which $\lambda_F$ exists. Then $$\lim_{z\to 1^-}\frac{F(z)}{F(z^k)}=\lambda_F.$$
\end{proposition}

\begin{proof} This follows directly from Theorem \ref{asymp}. To see this, note that using the identity  $$1-z^k=(1-z)(1+z+\cdots+z^{k-1})=(1-z)k(1+o(1)),$$ valid as $z\to 1^-$,  applying Theorem \ref{asymp} to $F(z^k)$ gives $$F(z^k)=\frac{C(z^k)}{(1-z^k)^{\log_k \lambda_F}} (1+o(1))=\frac{1}{k^{\log_k\lambda_F}}\cdot\frac{C(z)}{(1-z)^{\log_k \lambda_F}} (1+o(1)).$$ Thus $$\frac{F(z)}{F(z^k)}=k^{\log_k\lambda_F}(1+o(1))=\lambda_F(1+o(1)),$$ which is the desired result.
\end{proof}

This proposition gives the following corollary.

\begin{corollary}\label{cor} Let $F(z)\in\B{R}[[z]]$ be a $k$-Mahler function, converging inside the unit disc, for which $\lambda_F$ exists. Then $\lambda_F\in\B{R}$.
\end{corollary}

Corollary \ref{cor} should be compared to Perron's Theorem and the Perron-Frobenius Theorem regarding the spectral radius of a real matrix with positive or nonnegative entries; see \cite{F1908, F1909, F1912, P1907a, P1907b}.

\section{Universal test for transcendence of Mahler functions}\label{sec3}

As recalled in the Introduction, a result of B\'ezivin \cite{B1994} allows one to determine transcendence of a Mahler function by proving only irrationality. So let's suppose that we have a rational solution to \eqref{MFE}. What can we say then? Our first result of this section gives bounds on the degrees of the numerator and the denominator of a rational Mahler function.

\begin{proposition} Let $F(z)=P(z)/Q(z)$ be a rational $k$-Mahler function satisfying \eqref{MFE} with $\gcd(P(z),Q(z))=1$ and set $H:=\max\{\deg a_i(z):i=0,\ldots,d\}.$ Then $$\deg Q(z)\leqslant \lfloor H(k-1)/(k^{d+1}-2k^d+1)\rfloor,$$ and $$\deg P(z)\leqslant \deg Q(z)+\lfloor H/k^{d-1}(k-1)\rfloor.$$
\end{proposition}

\begin{proof} Write $F(z)=P(z)/Q(z)$ with $\gcd(P(z),Q(z))=1$.  Since $F(z)$ is a power series, $Q(0)\neq 0$.  
Then we have
$$\sum_{i=0}^d a_i(z)P(z^{k^i})/Q(z^{k^i})=0.$$  In particular, if we multiply both sides by 
$$R(z):=\prod_{j=0}^{d-1} Q(z^{k^j}),$$ we see that
$Q(z^{k^d})$ must divide $a_d(z)P(z^{k^d})R(z)$.  Since $\gcd(P(z),Q(z))=1$, we then have
that $Q(z^{k^d})$ divides $a_d(z)R(z)$.  Let $D$ denote the degree of $Q(z)$.  Then considering degrees, we have
$$k^d D \leqslant \deg a_d(z) + \deg R(z) \leqslant H + D+kD +\cdots + k^{d-1}D.$$
In other words, $(k^d-k^{d-1}-\cdots -1) D \leqslant H$.  Since
$$k^d - k^{d-1}-\cdots - 1 =  k^d - (k^d-1)/(k-1) \geqslant k^d (k-2)/(k-1),$$ if $k>2$, we have 
$$D\leqslant H(k-1)/k^d(k-2).$$  If $k=2$, then all we get is $D\leqslant H$.  In any case, setting
$$A(H,k,d):=\lfloor H(k-1)/(k^{d+1}-2k^d+1)\rfloor,$$ we have $D=\deg Q(z)\leqslant A(H,k,d)$.

Similarly, we can bound the degree of $P(z)$, but this is slightly more subtle.
Suppose that $F(z)=P(z)/Q(z)$ has a pole at $z=\infty$ of order $M>0$ with $Mk^{d-1} + H < M k^{d}$.  Since $F(z)$ satisfies \eqref{MFE}, we have
\begin{equation}\label{MFEminus}F(z^{k^d})a_d(z)=-\sum_{i=0}^{d-1} a_i(z) F(z^{k^i}).\end{equation} Now, the right-hand side of \eqref{MFEminus} has a pole at $z=\infty$ of order at most
$k^{d-1}M+H$ and the left-hand side of \eqref{MFEminus} has a pole at $z=\infty$ of order at least $k^d M$.  Since the equality \eqref{MFEminus} must hold, we conclude that $Mk^{d-1}+H\geqslant Mk^d$ and so $M\leqslant H/(k^d-k^{d-1}).$  In other words, \begin{equation*}\deg P(z)\leqslant \deg Q(z)+H/k^{d-1}(k-1).\qedhere\end{equation*}  
\end{proof}

While we can bound the degrees of the numerator and the denominator of a rational Mahler function, unfortunately, deciding whether a general power series is a rational function is still not effectively determinable. After all, one can imagine that the function is very close to some rational function and one must go very far out when looking at its coefficients to see that it is irrational. Fortunately, as we now show, deciding whether a Mahler function is a rational function is effective.

\begin{lemma}\label{zeros} Let $F(z)$ be a Mahler function satisfying \eqref{MFE}. If $P(z)/Q(z)$ is a rational function with $Q(0)\neq 0$ and the degrees of $P(z)$ and $Q(z)$ are strictly less than some positive integer $\kappa$, then 
$F(z)-P(z)/Q(z)$ is either identically zero or it has a nonzero coefficient of $z^i$ for some $i\leqslant H+\kappa\cdot k^{d+1}/(k-1)$. 
\end{lemma}

\begin{proof} Suppose not. Then $F(z)-P(z)/Q(z)=z^M T(z)$ for some nonzero power series $T(z)$ with $T(0)$ nonzero and some $M>H+\kappa\cdot k^{d+1}/(k-1)$.  Then we have
\begin{equation}\label{FEpq}\sum_{i=0}^d a_i(z) P(z^{k^i})/Q(z^{k^i}) = \sum_{i=0}^d a_i(z) z^{Mk^i}T(z^{k^i}).\end{equation}
Notice the right-hand side of \eqref{FEpq} has a zero of at least order $M$ at $z=0$.  On the other hand, we can write the left-hand side of \eqref{FEpq} as a rational function with denominator $Q(z)Q(z^k)\cdots Q(z^{k^d})$ and numerator
$$\sum_{i=0}^d a_i(z) P(z^{k^i}) R_i(z),$$ where
$R_i(z):=\prod_{j\neq i} Q(z^{k^j})$.  Thus the numerator of the left-hand side of \eqref{FEpq} when written in lowest terms has degree at most $H+\kappa (k^d+\cdots +k+1)$.  But this can occur only if the left-hand side of \eqref{FEpq} is identically zero since $M>H+\kappa (k^{d+1}-1)/(k-1)$, a contradiction.
\end{proof}

We can now show that our universal transcendence test (see Figure \ref{Fig2} in the Introduction) is valid.

\begin{proof}[Proof of Universal test for transcendence of Mahler functions] Let ${\bf M}$ be the matrix formed in Step 2 of our universal transcendence test described in Figure~\ref{Fig2}. 

Suppose that ${\bf M}$ does not have full rank.  Then there is a nonzero row vector
${\bf q}:=[q_0,q_1,\ldots ,q_{\kappa}]$ such that ${\bf q}\cdot {\bf M}=0$.  In other words,
$$(q_{\kappa}+q_{\kappa-1} z + \cdots + q_0 z^{\kappa})F(z)$$ has the property that $0$ is the coefficient of $z^i$ for $i=\kappa,\ldots , \kappa+H+\kappa (k^{d+1}-1)/(k-1)$; that is, there is a polynomial $P(z)$ of degree less than $\kappa$ such that 
$$(q_{\kappa}+q_{\kappa-1} z + \cdots + q_0 z^{\kappa})F(z)-P(z)$$ has a zero of order at least
$\kappa+H+\kappa (k^{d+1}-1)/(k-1)$ at $z=0$. Then $P(z)$ must have an order of zero at $z=0$ that is at least as great as the order of zero of $Q(z):=q_{\kappa}+q_{\kappa-1} z + \cdots + q_0 z^{\kappa}$ at $z=0$.  This means that 
$P(z)/Q(z)$ can be reduced to be written as a ratio of polynomials of degree less than $\kappa$ with the denominator being nonzero at $z=0$ and such that $F(z)-P(z)/Q(z)$ has a zero at $z=0$ of order at least
$H+\kappa (k^{d+1}-1)/(k-1)$. Lemma \ref{zeros} gives then that $F(z)-P(z)/Q(z)$ is identically zero and hence $F(z)$ is rational. 

Conversely, if $F(z)$ is rational, then we write $F(z)=P(z)/Q(z)$ with the degree of $P(z)$ and $Q(z)$ less than $\kappa$ and use $Q(z)$ to provide a nonzero row vector ${\bf q}$ as above with ${\bf q}\cdot{\bf M}=0$.
\end{proof} 

\section{Some examples and remarks}\label{sec4}

Before giving some examples using both our tests, we consider the complexity of our universal transcendence test. Notice the functional equation \eqref{MFE} for $F(z)$ allows one to compute the coefficients of $F(z)$ up to $z^n$ in ${\rm O}(n)$ steps.  Row reduction of an $m\times n$ matrix can certainly be done in ${\rm O}(mn^2)$ steps, so in principle our universal transcendence test can be done in ${\rm O}(\kappa^3 k^{2d})={\rm O}(H^3 k^{2d})$ operations, where $\kappa$ and $H$ are as defined in Figure \ref{Fig2}. Of course, for small Mahler degrees, our eigenvalue test (when applicable) can be very fast, and sometimes even done by inspection (e.g., when $F(z)$ is a degree-$2$ Mahler function using the quadratic formula).

To begin our examples, we note that our eigenvalue transcendence test (via Theorem \ref{main}) implies the following general result for degree-$1$ Mahler functions. 

\begin{corollary}\label{main2}
Let $k\geqslant 2$ be a positive integer and $r(z)\in\B{C}(z)$ be convergent inside the unit disc with $r(1)$ defined and not equal to $k^n$ for some $n\in\B{Z}$. Then the infinite product $$F(z):=\prod_{n\geqslant 0}r(z^{k^n})$$ is transcendental over $\B{C}(z)$. 
\end{corollary}

Setting $r(z)=(1+z+z^2)/z$, Corollary \ref{main2} provides a new proof of the transcendence of the generating function of the Stern sequence, which was first proved in~\cite{C2010}. 

We proceed with some examples of degree-$2$ and degree-$3$ Mahler functions. 

\begin{example} The canonical examples of degree-$2$ Mahler functions are the pair $F(z)$ and $G(z)$ defined by Dilcher and Stolarsky~\cite{DS2009}. That is, let $F(z)$ and $G(z)$ be the two $2$-Mahler functions of degree $2$ with coefficients in $\{0,1\}$ for which $$F(z)-(1+z+z^2)F(z^4)+z^4F(z^{16})=0,$$
and $$zG(z)-(1+z+z^2)G(z^4)+G(z^{16})=0.$$ 
Using our eigenvalue transcendence test, we note that in both cases the characteristic polynomial is $p_F(\lambda)=p_G(\lambda)=\lambda^2-3\lambda+1,$ which has roots $(3\pm\sqrt{5})/3$. As neither of these roots is a power of $2$, both $F(z)$ and $G(z)$ are transcendental. This result was previous proved by us \cite{C2010} and independently by Adamczewski \cite{A2010}.

Repeating using our universal transcendence test for $F(z)$ as defined in the previous paragraph, we have $k=4$, $d=2$, and $H=4$. We get 
$$\kappa =\lfloor 12/33 \rfloor +\lfloor 4/12\rfloor +1=1.$$ Thus we must check the rank of the $2\times 26$ matrix ${\bf M}$ and see whether or not it has rank $2$. Computing just first three columns we see that $${\bf M}=\left[\begin{array}{rrrr} 1&1&1&\cdots \\ 1&1&0&\cdots \end{array}\right],$$ and so ${\bf M}$ has full rank and thus $F(z)$ is transcendental. Our universal transcendence test for the function $G(z)$ is done similarly.
\end{example}

We provide two more results of degree-$2$ Mahler functions that are of interest to communities in theoretical computer science and combinatorics on words. Both examples contain $2$-automatic sequences whose generating functions are $2$-Mahler functions. In both cases, transcendence is known, and can be deduced from results of Adamczewski and Bugeaud \cite{AB2007}, Adamczewski and Rivoal \cite{AR2009}, or even from more na\"ive approaches.

\begin{example} The {\em Rudin-Shapiro sequence} is given by the recurrences $r_0=1,$ $r_{2n}=r_n$, $r_{4n+1}=r_n$, and $r_{4n+3}=-r_{2n+1}.$ The sequence $\{r_n\}_{n\geqslant 0}$ is $2$-automatic and as a consequence of the above relations, its generating function $R(z)=\sum_{n\geqslant 0}r_n z^n$ satisfies the $2$-Mahler equation $$R(z)-(1-z)R(z^2)-2zR(z^4)=0.$$ The function $R(z)$ has the  characteristic polynomial $p_R(\lambda)=\lambda^2-2.$ Since the values of $r_n$ are $\pm1$, we have that the root corresponding to $R(z)$ is $\lambda_R=\sqrt{2}$, which is not a power of $2$. Transcendence of $R(z)$ now follows from Theorem \ref{main}.

To apply our universal transcendence test, we see $k=2$, $d=2$, $H=1$, and we get $\kappa =2$.  Thus we look at the associated $3\times 16$ matrix and determine whether or not it has rank $3$. Looking at the first $3$ columns, we have
$${\bf M}= \left[\begin{array}{rrrr} 1 & 1 & 1 & \cdots \\ 1 & 1 & -1 &\cdots \\ 1&-1 & 1&\cdots \end{array}\right],$$
which has full rank and so we have transcendence of $R(z)$.
\end{example}

\begin{example} The {\em Baum-Sweet sequence} is given by the recurrences $b_0=1,$ $b_{4n}=b_{2n+1}=b_n$, and $b_{4n+2}=0.$ The sequence $\{b_n\}_{n\geqslant 0}$ is also $2$-automatic and as a consequence of the above relations, its generating function $B(z)=\sum_{n\geqslant 0}b_n z^n$ satisfies the $2$-Mahler equation $$B(z)-zB(z^2)-B(z^4)=0.$$ The function $B(z)$ has the characteristic polynomial $p_B(\lambda)=\lambda^2-\lambda-1.$ Since the values of $b_n$ are $0$ or $1$, we have that $\lambda_B=(1+\sqrt{5})/2$.  Theorem \ref{main} implies that $B(z)$ is transcendental.

Our universal transcendence test is completed similar to the Rudin-Shapiro example; we leave it for the interested reader.
\end{example}

The following is an example concerning a Mahler function of degree~$3$. This example seems to be the first transcendence result in the literature for a Mahler function of degree larger than $2$.

\begin{example} Consider the following example of Dumas \cite[Example~64]{D1993These} of a $2$-Mahler function of degree $3$; that is, consider the function $H(z)$ beginning $$H(z)=z+z^2+2z^3+z^4+5z^5+2z^6+7z^7+z^8+9z^9+5z^{10}+\cdots,$$ which satisfies \begin{multline*} z^3H(z)+z^2(2z^4-2z^2-z-2)H(z^2)\\-(4z^8+5z^6+z^4-2z^2-1)H(z^4)\\+(4z^8+2z^4-1-6z^{12})H(z^8)=0.\end{multline*} The function $H(z)$ has the characteristic polynomial $p_H(\lambda)=\lambda^3-3\lambda^2-7\lambda-1$, which has distinct roots. Moreover, we see that $\lambda_H\approx 4.577089445$, which is not a power of $2$, and so $H(z)$ is transcendental.

To apply our universal transcendence test here one has $H=12$, $k=2$, $d=3$, and so since $\kappa=15$ one needs to compute a $16\times 253$ matrix.  While, as before, one could probably get by with many fewer columns, our eigenvalue test works much faster in this case.  
\end{example}

We end our paper with the canonical example of a degree-$1$ Mahler function, and a function for which our eigenvalue transcendence test is not applicable.

\begin{example} The {\em Thue-Morse sequence} $\{t(n)\}_{n\geqslant 0}$ is given by $t(n)=(-1)^{s_2(n)}$, where $s_2(n)$ is the number of ones in the binary expansion of $n$. The generating function $T(z)=\sum_{n\geqslant 0} t(n)z^n$ satisfies $$T(z)-(1-z)T(z^2)=0.$$ Here $p_T(z)=\lambda+0,$ so since $a_1=0$, our eigenvalue transcendence test cannot be applied.

To apply our universal test, we note that $d=1$, $k=2$, and $H=1$. So $\kappa=2$, and we compute the rank of the $3\times 9$ matrix $${\bf M}=\left[\begin{array}{rrrrrrrrr} -1&1&-1&1&1&-1&-1&1&1 \\ 1&-1&1&1&-1&-1&1&1&-1 \\ -1&1&1&-1&-1&1&1&-1&1\end{array}\right].$$ Here one need only consider columns $2$ through $4$ to see that ${\bf M}$ has full rank, and indeed $T(z)$ is transcendental.
\end{example}

\bibliographystyle{amsplain}
\def\cprime{$'$}
\providecommand{\bysame}{\leavevmode\hbox to3em{\hrulefill}\thinspace}
\providecommand{\MR}{\relax\ifhmode\unskip\space\fi MR }
\providecommand{\MRhref}[2]{%
  \href{http://www.ams.org/mathscinet-getitem?mr=#1}{#2}
}
\providecommand{\href}[2]{#2}


\end{document}